\documentclass[reqno,A4paper]{amsart}
\usepackage{amsthm}
\usepackage{amsmath}
\usepackage{amsfonts}
\usepackage[utf8]{inputenc}
\usepackage[T1]{fontenc}
\usepackage{graphicx}
\usepackage{enumerate}
\usepackage{array}
\usepackage[mathscr]{eucal}
\usepackage{ifpdf}
\usepackage{mathtools}
\ifpdf
\usepackage[all,pdf,2cell]{xy}\UseAllTwocells\SilentMatrices
\else
\usepackage[all,xdvi,2cell]{xy}\UseAllTwocells\SilentMatrices
\fi
\newcommand{\id}{\mathrm{id}}
\newcommand{\C}{\mathcal{C}}

\newcommand{\E}{\mathcal{E}}

\newcommand{\PSTS}{\mathbf{PSTS}}
\newcommand{\Graph}{\mathbf{Graph}}
\newcommand{\Perf}{\mathbf{Perf}}

\newcommand{\Set}{\mathbf{Set}}
\newcommand{\isleftadjoint}{\dashv}

\newtheorem{theorem}{Theorem}[section]
\newtheorem{corollary}[theorem]{Corollary}
\theoremstyle{definition}
\newtheorem{definition}[theorem]{Definition}
\newtheorem{example}[theorem]{Example}
\numberwithin{equation}{section}
\begin{document}
\title{Two monads on the category of graphs}
\author{Gejza Jen\v ca}
\newcommand{\acr}{\newline\indent}
\address{\llap{*\,}
	Department of Mathematics and Descriptive Geometry\acr 
	Faculty of Civil Engineering\acr
	Radlinsk\' eho 11, 81368 Bratislava\acr 
	SLOVAKIA}
\email{gejza.jenca@stuba.sk}
\thanks{
This research is supported by grants VEGA 2/0069/16, 1/0420/15,
Slovakia and by the Slovak Research and Development Agency under the contracts
APVV-14-0013, APVV-16-0073}
\begin{abstract}
We introduce two monads on the category of graphs and prove that their Eilenberg-Moore
categories are isomorphic to the category of perfect matchings and the category
of partial Steiner triple systems, respectively. 
As a simple application of these
results, we describe the product in the categories of perfect matchings
and partial Steiner triple systems.
\end{abstract}

\maketitle

\section{Introduction}

Despite of the fact that there is a considerable amount of literature about graph
homomorphisms and their properties, the attempts to look at graphs from
the viewpoint of category theory appear to be rather rare.  

In the present note we prove that two classical notions of graph theory arise
as instances of a category-theoretic notion of an {\em algebra for a monad}:
the notion of a {\em perfect matching} on a graph and the notion of a {\em
partial Steiner triple system}. In addition, we describe the product in the
categories of perfect matchings and partial Steiner triple systems.

\section{Preliminaries}

For here undefined notions of category theory we refer to introductory books
\cite{awodey2006category,riehl2016category}; both include a chapter devoted to
monads and their algebras. Alternatively, see the classic book \cite{mac1998categories}. 

\subsection{The category $\Graph$}

In this paper we shall deal solely with simple, loopless and undirected graphs. 
If $G$ is a graph, we write $V(G)$ for the (possibly empty) set of vertices of
$G$ and $E(G)$ for the set of edges of $G$. The edges of a graph $G$ are identified with
two-element sets of vertices of $G$.
The adjacency relation on the set of vertices of a graph $G$ is denoted by $\sim_G$. 
We shall mostly drop the subscript $G$ on $\sim_G$,
whenever there is no danger of confusion.

For two graphs $G$ and $H$, a {\em a morphism of graphs} $f:G\to H$ 
is a mapping of sets $f:V(G)\to V(H)$ such that, for all 
$u,v\in V(G)$, $u\sim v$ implies that $f(u)\sim f(v)$. The morphisms of
graphs are usually called {\em homomorphisms} \cite{hahn1997graph}.

Clearly, the composition of two morphisms is a morphism and, for every graph
$G$, the identity map $\id_G$ on $V(G)$ is a morphism. So the class of all
graphs equipped with morphisms of graphs forms a category, which we will call
$\Graph$.

Most of the categorically-minded authors prefer to deal with graphs
that admit loops, see for example \cite{brown2008graphs}. The reason
for that decision is probably that the category $\Graph$ does not have all coequalizers,
so it is not cocomplete.
Perhaps surprisingly, lack of loops in $\Graph$ will be necessary prove the
main results of the present note.

\subsection{Monads and their algebras}

A monad on a category $\C$ is a triple $(T,\eta,\mu)$,
where $T$ is an endofunctor on $\C$ and $\eta:\id_C\to T$, $\mu:T\circ T\to T$
are natural transformations of endofunctors on $\C$ such that,
$\mu\circ(T\eta)=\mu\circ(\eta T)=\id_T$ and $\mu\circ(T\mu)=\mu\circ(\mu T)$.
That means, for every object $A$ of $\C$, the diagrams
\begin{equation}
\label{diag:monad}
\xymatrix{
T(A)
	\ar[r]^-{T(\eta_A)}
	\ar[rd]_-{\id_{T(A)}}
&
T^2(A)
	\ar[d]^-{\mu_A}
&
T(A)
	\ar[l]_-{\eta_{T(A)}}
	\ar[ld]^-{\id_{T(A)}}
\\
~
&
T(A)
}
\qquad
\xymatrix{
T^3(A)
	\ar[r]^{T(\mu_A)}
	\ar[d]_{\mu_{T(A)}}
&
T^2(A)
	\ar[d]^{\mu_A}
\\
T^2(A)
	\ar[r]^{\mu_A}
&
T(A)
}
\end{equation}
commute.

In what follows, we shall mostly drop $\mu$ and $\eta$ from the signature of the monad
and simply write ``a monad $T$ on $\C$''.

\begin{example}
\label{ex:monoidmonad}
For a set $X$, a {\em word over the alphabet $X$} is
a finite sequence of elements of $X$, possibly empty. To avoid
notational ambiguities, we enclose words in square brackets.
For example, $[12023]$, $[7]$ and $[\,]$ are three words over the alphabet $\mathbb N$.

Consider an endofunctor $M$ on the category of sets $\Set$ that takes a set $X$
the set of all words over the alphabet $X$ and a mapping
$f:X\to Y$ to a mapping $M(f):M(X)\to M(Y)$ that operates on words characterwise:
$$
M(f)([x_1\dots x_n])=[f(x_1)\dots f(x_n)].
$$
For a set $X$, $\eta_X:X\to M(X)$ is given by the rule $\eta_X(x)=[x]$ (the 
one-letter word) and the mapping $\mu_X:M^2(X)\to M(X)$ concatenates the
inner words:
$$
\mu_X([[x_1^1\dots x_{n_1}^1][x_1^2\dots x_{n_2}^2]\dots[x_1^k\dots x_{n_k}^k]])=
[x_1^1\dots x_{n_1}^1 x_1^2\dots x_{n_2}^2 \dots x_1^k\dots x_{n_k}^k].
$$
Then $M$ is a monad on $\Set$, we call it the {\em free monoid monad}.
\end{example}

Note that, for every set $X$, $M(X)$ is (the underlying set of) the free monoid
freely generated by $X$, that means the set of all monoid terms (words) over
the set $X$. The $\eta_X$ map ``embeds the variables'' into terms and the
$\mu_X$ map ``evaluates a term over the free monoid''. The construction
generalizes, in a straightforward way, if we replace monoids with other
equational classes (or varieties) of universal algebras
\cite{Gra:UA}, for example groups or semilattices. 

Although monads are interesting and useful structures per se 
(see for example the seminal paper \cite{moggi1991notions} for applications in the theory
of programming), the main use of monads is in their
deep connection with the notion of an adjoint pair of functors.

On one hand, every adjunction $F\isleftadjoint G$ induces a monad on the domain
category of $F$. On the other hand,
every monad induces an adjunction between the underlying category $\C$ and another category
$\C^T$. The objects of $\C$ are called {\em algebras for
the monad $T$} and (in many cases) can be described as ``an object of $\C$
equipped with an additional operations-like structure''. The morphisms of $\C^T$ can be then thought of as
``the $\C$-morphisms preserving the additional structure''.

Let $(T,\eta,\mu)$ be a monad on a category $\C$.
The {\em category of algebras for $(T,\eta,\mu)$}, also known as the 
{\em Eilenberg-Moore category for $(T,\eta,\mu)$} is a category
(denoted by $\C^T$), such that  
objects (called {\em algebras for the monad $T$}) of $\C^T$ are pairs 
$(A,\alpha)$, where $\alpha:T(A)\to A$, such that the diagrams
\begin{equation}
\label{diag:algtriangle}
\xymatrix{
A \ar[r]^{\eta_A} \ar[rd]_{1_A} & T(A) \ar[d]^{\alpha} \\
 & A
}
\end{equation}
\begin{equation}
\label{diag:algsquare}
\xymatrix{
T^2(A) \ar[r]^{T(\alpha)} \ar[d]_{\mu_A} & T(A) \ar[d]^{\alpha}\\
T(A) \ar[r]_{\alpha} & A
}
\end{equation}
commute. A morphism of algebras 
$h:(A_1,\alpha_1)\to (A_2,\alpha_2)$ is a $\C$-morphism $h:A_1\to A_2$ such that
the diagram
\begin{equation}
\label{diag:algmorphism}
\xymatrix{
T(A_1) \ar[r]^{T(h)} \ar[d]_{\alpha_1}& T(A_2)\ar[d]^{\alpha_2}\\
A_1\ar[r]^{h}&A_2
}
\end{equation}
commutes.

The adjunction mentioned above between $\C$ and $\C^T$ is given by a pair of
functors $F\isleftadjoint U$. The forgetful right adjoint functor $U:\C^T\to\C$ maps an algebra
$(A,\alpha)$ to its underlying object, $U(A,\alpha)=A$ and a morphism of
algebras to its underlying $\C$-morphism. The left adjoint $F:\C\to\C^T$ maps
an object $A$ of $\C$ to the pair $(T(A),\mu_A)$. This pair is always an
algebra for the monad $T$. Such adjunctions (and adjunctions equivalent to them) 
are called {\em monadic}.

\begin{example}
The category of algebras $\Set^M$ for the free monoid monad $M$ from the 
Example~\ref{ex:monoidmonad} is isomorphic to the category of monoids. Indeed, if $A$ is a set and
$(A,\alpha)$ is an algebra for the monad $M$, then we may equip $A$ with a binary
operation $*:A\times A\to A$ given by the rule $a_1*a_2:=\alpha([a_1a_2])$ and a constant 
$e\in A$ given by $e:=\alpha([\,])$. One can then easily check that $(A,*,e)$ is a
monoid. On the other hand, every monoid $(A,*,e)$ induces an algebra $(A,\alpha)$ for the
free monoid monad given by the evaluation of terms: $\alpha([a_1\dots a_n])=a_1*\dots a_n$.
These constructions can be easily shown to be functorial. Moreover,
they establish an equivalence of categories 
between $\Set^M$ and the category of monoids.
\end{example}

Again, this example generalizes to any equational class of algebras: every equational class
of algebras $\E$ is equivalent to the category $\Set^T$, where $T$ is the free $\E$-algebra
monad on $\Set$. Thus, categories of algebras for a monad generalize equational classes
of algebras and the notion can be used to extend parts of the theory of equational classes
of algebras to a more general context, by using a different category than $\Set$ as the underlying
construction.

Examples include:
\begin{itemize}
\item Topological groups are algebras for a monad on the category of topological spaces and continuous maps.
\item Modules over a fixed ring are algebras for a monad on the category of abelian groups.
\item Small categories are algebras for a monad on the category of directed multigraphs (or quivers).
\item Compact Hausdorff spaces are algebras for a monad on $\Set$ \cite{manes1974compact}.
\end{itemize}

\section{The perfect matching monad and its algebras}

In this section, we will introduce a monad $T$ on $\Graph$ and prove that
the category of algebras for $T$ is isomorphic to the category of
graphs equipped with a perfect matching. We call this monad a 
{\em perfect matching monad}.

Recall, that a {\em perfect matching} \cite{lovasz2009matching} on a graph $A$ is a set $M$ of edges of $A$ such that
no two edges in $M$ have a vertex in common and $\bigcup M=V(A)$.

In the present note it will be of advantage  
if we use an alternative definition, that is clearly
equivalent to the usual one.
\begin{definition}
Let $A$ be a graph. A {\em perfect matching} on $A$ 
is a mapping $m:V(A)\to V(A)$ such that, for all $x\in V(A)$,
$\{x,m(x)\}$ is an edge of $A$ and $m\circ m=id_{V(A)}$
\end{definition}

The {\em category of perfect matchings} is a category, where
\begin{itemize}
\item
the objects are all pairs $(A,m)$, where $m$ is a perfect matching on a graph $A$ and
\item
the morphisms $f:(A,m)\to (Y,m)$ are graph homomorphisms $f:A\to Y$, that preserve the
$m$, meaning that for all vertices $x$ of $A$, $f(m(x))=m(f(x))$.
\end{itemize}
The category of perfect matchings is denoted by $\Perf$.

\begin{figure}
\begin{center}
\includegraphics{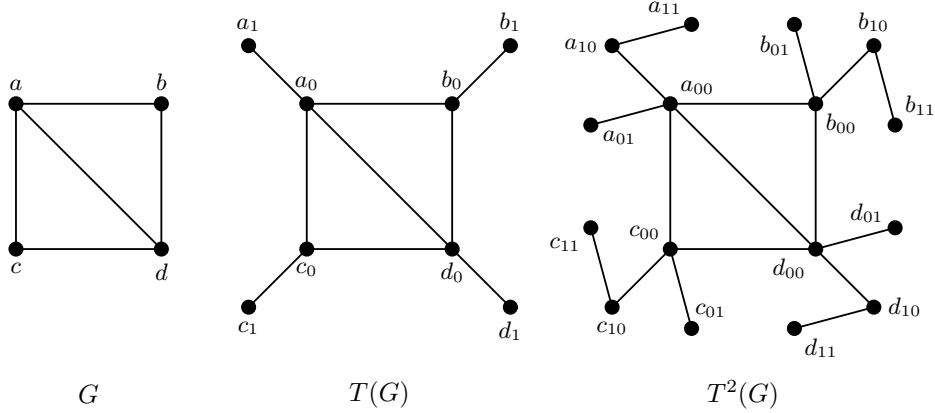}
\end{center}
\caption{The endofunctor $T$ adds a new pendant edge to every vertex.}
\label{fig:Tmonad}
\end{figure}

For a graph $A$, $T(A)$ is a graph that extends the graph $A$ by new leaf
vertices, attaching a new leaf (or a new pendant edge) vertex at every vertex
of $A$.

Let us introduce a notation that will turn out to be useful in our description
of the perfect matching monad. The set of vertices of $T(A)$ is the set
$V(A)\times\{0,1\}$; we denote the vertices of $T(A)$ by $x_i$, where $x\in
V(A)$, $i\in\{0,1\}$.  The vertices of $T(A)$ with $i=0$ mirror the original
vertices, the vertices with $i=1$ are the new leaves. The edges $T(A)$ are of
one of two types: either the edge is of the form $\{x_0,y_0\}$, where
$\{x,y\}\in E(A)$, or it is of the form $\{x_0,x_1\}$, where $x\in V(A)$.

It is easy to see that this construction is functorial, so $T$ is an
endofunctor on the category of graphs. Explicitly, if $f:A\to B$ is a graph
homomorphism, then $T(f):T(A)\to T(B)$ is given by the rule
$T(f)(x_i)=f(x)_i$.

Moreover, for every graph $A$ we clearly have a morphism $\eta_A:A\to T(A)$
given by the rule $\eta_A(x)=x_0$. There is another morphism $\mu_A:T(T(A))\to
T(A)$, given by the rule $\mu_A(v_{ij})=v_{i\oplus j}$, where $\oplus$ denotes
the exclusive or operation on the set $\{0,1\}$, also
known as the addition in the 2-element cyclic group $\mathbb Z_2$.
Note how this operation folds, in a natural way,
the new pendant edges $\{v_{00},v_{01}\}$ and $\{v_{10},v_{11}\}$ that were
added in the second iteration of $T$ onto the edge $\{v_0,v_1\}$.
These families of maps determine natural transformations
$\eta:\id_\Graph\to T$, $\mu:T^2\to T$.

\begin{theorem}
The triple $(T,\eta,\mu)$ is a monad on the category of graphs.
\end{theorem}
\begin{proof}
Let us prove that, for every graph $A$, the monad axioms
are satisfied. This is a consequence of the fact that
$(\{0,1\},\oplus,0)$ is a monoid. Indeed, let us check the validity of the 
associativity axiom. Let $x_{ijk}\in V(T^3(A))$. Then 
\begin{align*}
\bigl(\mu_A\circ\mu_T(A)\bigr)(v_{ijk})&=\mu_A\bigl(v_{i(j\oplus k)}\bigr)=v_{i\oplus(j\oplus k)}\text{ and }\\
\bigl(\mu_A\circ T(\mu_A)\bigr)(v_{ijk})&=\mu_A\bigl(v_{(i\oplus j)k}\bigr)=v_{(i\oplus j)\oplus k}.
\end{align*}
Similarly, the unit axioms are satisfied because $0$ is a unit for the operation $\oplus$.
\end{proof}

\begin{theorem}
The category $\Perf$ is isomorphic to the category of algebras for the monad $(T,\eta,\mu)$.
\label{thm:MatchingsAreAlgebras}
\end{theorem}
\begin{proof}
Let $\alpha:T(A)\to A$ be an algebra for $T$. The triangle diagram~\eqref{diag:algtriangle}
means that, for all $x\in V(A)$, $\alpha(x_0)=x$. So every algebra 
$\alpha$ is completely determined by its values on the new leaves $x_1$ of
$T(A)$. We claim that the mapping $m:V(A)\to V(A)$ given by the
rule $m(x)=\alpha(x_1)$ is a perfect matching on $A$. Indeed, since
$\alpha$ is a homomorphism of graphs, it must take the edge $\{x_0,x_1\}$ of $T(A)$
to an edge of $A$, so $\{\alpha(x_0),\alpha(x_1)\}=\{x,m(x)\}$ is an edge of $A$.
To prove that $m\circ m=id_{V(A)}$, consider the fact that the commutativity of the square diagram~\eqref{diag:algsquare}
means that, for all $x\in V(A)$ and $i,j\in\{0,1\}$,
$\alpha(\alpha(x_i)_j)=\alpha(x_{i\oplus j})$. In
particular, for $i=j=1$ we obtain
$$
m(m(x))=m(\alpha(x_1))=\alpha(\alpha(x_1)_1)=\alpha(x_{1\oplus 1})=\alpha(x_0)=x.
$$
This proves that $m$ is a perfect matching on $A$.

On the other hand, let $m$ be a perfect matching on $A$. Define a mapping 
$\alpha:V(T(A))\to V(A)$ by the rules $\alpha(x_0)=x$, $\alpha(x_1)=m(x)$.
We claim that $(A,\alpha)$ is an algebra for the monad $T$.
The fact that $\alpha$ is a graph homomorphism is easy to see. Clearly,
the triangle diagram~\eqref{diag:algtriangle} commutes.
Let us prove that~\eqref{diag:algsquare} commutes, that means, to prove the equality
$\alpha(\alpha(x_i)_j)=\alpha(x_{i\oplus j})$. For $i=0$,
$$
\alpha(\alpha(x_i)_j)=\alpha(\alpha(x_0)_j)=\alpha(x_j)=\alpha(x_{0\oplus j}).
$$
For $j=0$,
$$
\alpha(\alpha(x_i)_j)=\alpha(\alpha(x_i)_0)=\alpha(x_i)=\alpha(x_{i\oplus 0}).
$$
For $i=j=1$,
$$
\alpha(\alpha(x_1)_1)=\alpha(m(x)_1)=m(m(x))=x=\alpha(x_0)=\alpha(x_{1\oplus 1}),
$$
since $m$ is a perfect matching.

It remains to prove the these constructions are functorial and, as functors,
inverse to each other. The proof of this is completely straightforward and is
thus omitted.
\end{proof}
\begin{figure}
\begin{center}
\includegraphics{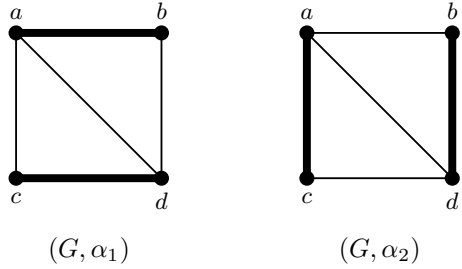}
\end{center}
\caption{Two perfect matchings of a square with a diagonal.}
\label{fig:match}
\end{figure}
\begin{example}
Let $G$ be the graph from Figure~\ref{fig:Tmonad}. It
has two perfect matchings, see Figure~\ref{fig:match}.
Under the isomorphism described in the proof of Theorem~\ref{thm:MatchingsAreAlgebras},
these two perfect matchings correspond to two algebras $(G,\alpha_1)$, $(G,\alpha_2)$
for the monad $T$ that are characterized by
$$
\begin{array}{|c||c|c|c|c|c|c|c|c|}
\hline
v & a_0 & b_0 & c_0 & d_0& a_1 & b_1 & c_1 & d_1\\
\hline
\alpha_1(v) &a &b &c &d & b & a & d & c \\
\alpha_2(v) &a &b &c &d & c & d & a & b \\
\hline
\end{array}
$$
\end{example}

Recall, that for two graphs $A,B$, their product $A\times B$
in the category $\Graph$ is a graph with vertex set $V(A\times B)=V(A)\times V(B)$ such that
for $(a_1,b_1),(a_2,b_2)\in V(A\times B)$ we have $(a_1,b_1)\sim_{A\times B}(a_2,b_2)$
if and only if $a_1\sim_A a_2$ and $b_1\sim_B b_2$. We write $p_A$ and $p_B$ for the
projections from $A\times B$ onto $A$ and $B$, respectively: $p_A(a,b)=a$, $p_B(a,b)=b$.

\begin{corollary}
Let $(A,m)$, $(B,m)$ be perfect matchings. Then their product
in $\Perf$ has the $A\times B$ as the underlying graph and the perfect matching on
$A\times B$ is given by the rule $m(a,b)=(m(a),m(b))$.
\end{corollary}
\begin{proof}
Since the functor $U:\Graph^T\to\Graph$ is a right adjoint in a
monadic adjunction, it creates all limits that exist in $\Graph$,
see \cite[Proposition IV.4.1]{maclane2012sheaves}.  For products
in $\Graph^T$ this means that $(A,\alpha)\times(B,\beta)$ is the
algebra $(A\times B,\gamma)$, where $\gamma$ is the unique arrow that makes
the diagram
\begin{equation}
\label{diag:product}
\xymatrix@C=4em{
T(A)
	\ar[d]_{\alpha}
&
T(A\times B)
	\ar[l]_-{T(p_A)}
	\ar@{.>}[d]^{\gamma}
	\ar[r]^-{T(p_B)}
&
T(B)
	\ar[d]^{\beta}
\\
A
&
A\times B
	\ar[l]^{p_A}
	\ar[r]_{p_B}
&
B
}
\end{equation}
commute. Explicitly, this means that $\gamma:T(A\times B)\to A\times B$ is
given by the rule $\gamma((a,b)_i)=(\alpha(a_i),\beta(b_i))$. If we
translate this fact into the language
of perfect matching via the isomorphism $\Graph^T\simeq\Perf$ that we constructed in the proof
of Theorem~\ref{thm:MatchingsAreAlgebras}, 
the product $(A,m)\times(B,m)$ is a matching on $A\times B$ is given by $m(a,b)=(m(a),m(b))$.
Indeed,
$$
m(a,b)=\gamma((a,b)_1)=(\alpha(a_1),\beta(b_1))=(m(a),m(b)).
$$
\end{proof}

Similarly, one can prove that an equalizer of two parallel morphisms
$f,g:(A,m)\to (B,m)$ in the category $\Perf$ is given by the restriction of $m$ to the induced
subgraph $E\xhookrightarrow{e} A$, where $V(E)=\{v\in V(A):f(v)=g(v)\}$. 

\section{The Steiner triple monad and its algebras}

\begin{figure}
\begin{center}
\includegraphics{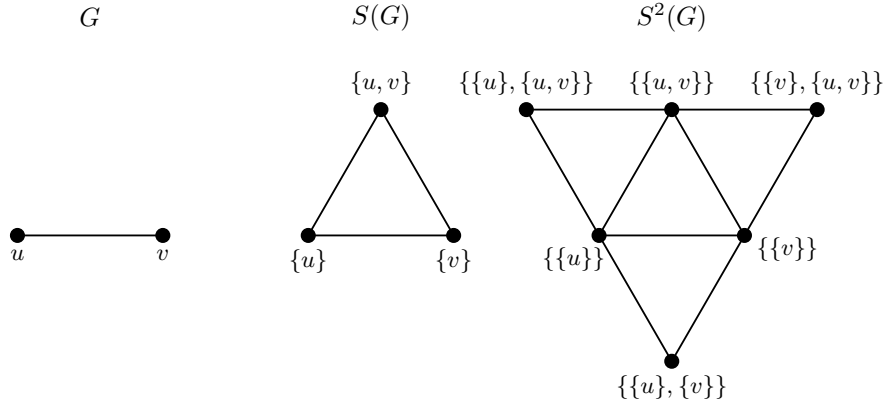}
\end{center}
\caption{The endofunctor $S$ adds a new triangle over every edge.}
\label{fig:steiner}
\end{figure}

In this section, we will introduce a monad $S$ on $\Graph$ and prove that
the category of algebras for $S$ is isomorphic to the category of
partial Steiner triple systems. We call this monad a 
{\em Steiner triple monad}.

Recall, that a {\em partial Steiner triple system} is a finite set $A$ equipped with a system of
3-element subsets $\Omega_A$ such that every 2-element subset $\{u,v\}$ of $A$ is in at most
one set in $T_A$. 
A partial Steiner triple system is {\em complete} if every 2-element
subset of $A$ occurs in exactly one set in $T_A$. A complete partial Steiner triple system is called
simply a {\em Steiner triple system}.

The category of partial Steiner triple systems has pairs $(A,\Omega_A)$ as objects. For 
a pair of objects $(A,\Omega_A)$ and $(B,\Omega_B)$ a morphism is a mapping $f:A\to B$ such that for 
every triple $\{u,v,w\}\in \Omega_A$, $\{f(u),f(v),f(w)\}\in\Omega_B$. We denote this
category by $\PSTS$.

Let us describe the monad $S$.  For a graph $G$, the graph $S(G)$ is the graph
that can be described as ``a copy of $G$ with a new triangle over every edge''.

Formally, it will be of advantage to define $S(G)$ as a graph
with the set of vertices $V(S(G))=\{\{u\}:u\in V(G)\}\cup E(G)$ and, for all $X,Y\in V(S(G))$, 
$X\sim_{S(G)}Y$ if and only if one of the following is true:
\begin{itemize}
\item $X=\{u\}$, $Y=\{v\}$ and $u\sim v$.
\item $X=\{u,v\}$, $Y=\{v\}$ and $u\sim v$.
\item $X=\{u\}$, $Y=\{u,v\}$ and $u\sim v$.
\end{itemize}
Clearly, $S$ is the object part of an endofunctor on the category $\Graph$; for
a morphism of graphs $f:G\to H$, $S(f):S(G)\to S(H)$ is given by
$S(\{u\})=f(\{u\})$ for $u\in V(G)$ and $S(\{u,v\})=S(\{f(u),f(v)\})$ for
$\{u,v\}\in E(G)$.

For every graph $G$, there is a morphism $\eta_G:G\to S(G)$ given by the rule
$\eta_G(v)=\{v\}$. It is obvious that this family of morphisms gives us a natural
transformation $\eta:\id_\Graph\to S$.

The last piece of data we need is a natural transformation $\mu:S^2\to S$. For every
edge $\{u,v\}$ of a graph $G$, $S(G)$ contains (essentially) the original edge and two new
edges, forming a triangle with vertices $\{u\},\{v\},\{u,v\}$. Repeating this construction
once again, we see that $S^2(G)$ consists of ``triangles with an inscribed triangle'', one
for every edge of the original graph $G$, see Figure~\ref{fig:steiner}. There is a clear candidate for the desired
mapping $\mu_G:S^2(G)\to S(G)$; $\mu_G$ just folds the three outer triangles
onto the inner one. Note that a formal description of $\mu_G$ is very simple:
the vertex set $V(S^2(G))$ consists of certain systems of sets of vertices of $G$. Then
$\mu_G(\mathbb X)$ is simply the symmetric difference of all sets in the system 
$\mathbb X$; in symbols $\mu_G(\mathbb X)=\Delta\mathbb X$. For example
$\mu_G(\{\{u\},\{u,v\}\})=\{u\}\Delta\{u,v\}=\{v\}$ and $\mu_G(\{\{u,v\}\})=\{u,v\}$. 

It is obvious
that $\mu_G$ is a morphism in $\Graph$ and that the family of morphisms
$(\mu_G)_{G\in\Graph}$ is a natural transformation from $S^2$ to $S$.

\begin{theorem}
$(S,\mu,\eta)$ is a monad on the category of graphs.
\end{theorem}
\begin{proof}
We need to prove that, putting $T=S$ in~\eqref{diag:monad},  both triangles and the square 
commute.

Let $\{u,v\}$ be a vertex of $S(G)$, that means $u\sim v$ in $G$.
Then, in particular, $u\neq v$ and we may chase $\{u,v\}$ around the triangles:
\begin{align*}
&\mu_G(S(\eta_G)(\{u,v\}))=\mu_G(\{\eta_G(u),\eta_G(v)\})=\mu_G(\{\{u\},\{v\}\})=
\{u\}\Delta\{v\}=\{u,v\}\\
&\mu_G(\eta_{S(G)}(\{u,v\}))=\mu_G(\{\{u,v\}\})=\{u,v\}
\end{align*}
The case of a vertex of the form $\{u\}$ is trivial.

The fact that the square in~\eqref{diag:monad} commutes follows from the fact that
the symmetric difference $\Delta$ is a commutative and associative operation 
with a neutral element $\emptyset$ and that $Y\Delta Y=\emptyset$.

In detail, every element of $V(S^3(G))$ is some set of sets of of sets of the form
\begin{equation}
\label{eq:delta}
\{\{X^1_1,X^1_2,\dots,X^1_{k_1}\},\{X^2_1,X^2_2,\dots,X^2_{k_2}\},\dots,\{X^n_1,X^n_2,\dots,X^n_{k_n}\}\}
\end{equation}
where each $X^i_j$ is a set of vertices of $G$. (In fact, $n\in\{1,2\}$,
each $k_i\in\{1,2\}$ and every $X_i^j$ is either a singleton or a pair, but we shall not need any of these facts.)
The morphism $S(\mu_G)$ maps this
element to the system of sets
$$
\{X^1_1\Delta X^1_2\Delta \dots\Delta X^1_{k_1},X^2_1\Delta X^2_2\Delta \dots\Delta X^2_{k_2},\dots,X^n_1\Delta X^n_2\Delta \dots\Delta X^n_{k_n}\}
$$
in $V(S^2(G))$ and this is mapped by $\mu_G$ to the set
\begin{equation}\label{eq:result}
(X^1_1\Delta X^1_2\Delta \dots\Delta X^1_{k_1})\Delta(X^2_1\Delta X^2_2\Delta \dots\Delta X^2_{k_2})\Delta\dots\Delta(X^n_1\Delta X^n_2\Delta \dots\Delta X^n_{k_n})
\end{equation}
Chasing the element~\eqref{eq:delta} the other way around the square~\eqref{diag:monad}, 
$\mu_{S(G)}$ maps it to
$$
\{X^1_1,X^1_2,\dots,X^1_{k_1}\}\Delta \{X^2_1,X^2_2,\dots,X^2_{k_2}\}\Delta\dots\Delta\{X^n_1,X^n_2,\dots,X^n_{k_n}\}
$$
this amounts to keeping just those sets $X^i_j$ that occur odd number of times. We then apply $\mu_G$ to the
resulting system of sets, so we get an expression exactly like~\eqref{eq:result}, but with
those $X^i_j$ that occur even number of times removed. However, removing those
sets does not change the value of the expression and we have proved that the square commutes.
\end{proof}

\begin{theorem}
The category of algebras $\Graph^S$ for the Steiner triple monad is isomorphic to the category
$\PSTS$.
\end{theorem}
\begin{proof}
Let $(G,\alpha)$ be an algebra for the Steiner triple monad. 
There are two types of vertices in $S(G)$: singletons and pairs. The triangle
axiom~\eqref{diag:algtriangle} tells us that $\alpha(\{u\})=u$ for all $u\in V(G)$.
Thus, every algebra $\alpha$ is completely determined by its value on pairs, that means, edges of $G$.
Since $\alpha$ is a morphism of graphs, the image of every triangle in $S(G)$ under $\alpha$ is a triangle
in $G$. Hence for every $\{u,v\}\in E(G)$, $\{\{u\},\{v\},\{u,v\}\}$ is a triangle in
$S(G)$ and its image is the triangle
$$
\{\alpha(\{u\}),\alpha(\{v\}),\alpha(\{u,v\})\}=\{u,v,\alpha(\{u,v\})\}.
$$
Thus, for every edge $\{u,v\}$ of $G$, $\alpha$ selects a triangle in $G$ with vertices
$u,v$ and $\alpha(\{u,v\})$. 

We claim that the set $V(G)$ equipped with the system of triples
$$
\Omega_\alpha=\bigl\{\{u,v,\alpha(\{u,v\})\}:u\sim_G v\bigr\}
$$
is a partial Steiner triple system. 

It is clear that a pair of vertices $\{u,v\}$ of $G$ occurs as a subset in at least one
of the triples in $\Omega_\alpha$ if and only if $\{u,v\}$ is an edge of
$G$. It remains to prove that every edge occurs in exactly one triple in
$\Omega_\alpha$ or, in other words, that the triple that arises from an edge
$\{u,v\}$ is the same as the triple that arises from the edge
$\{u,\alpha(\{u,v\})\}$.

In terms of properties of the $\alpha$ mapping, this amounts to
$$
\alpha(\{u,\alpha(\{u,v\}))=v.
$$
However, this follows by the commutativity of the square~\eqref{diag:algsquare}, because
$$
(\alpha\circ S(\alpha))(\{\{u\},\{u,v\}\})=\alpha(\{\alpha(\{u\}),\alpha(\{u,v\})\})=\alpha(\{u,\alpha(\{u,v\}))
$$
and
$$
(\alpha\circ\mu_G)(\{\{u\},\{u,v\}\})=\alpha(\{u\}\Delta\{u,v\})=\alpha(\{v\})=v.
$$
To prove that this construction is functorial,
let us write $F(G,\alpha)$ for the partial Steiner triple system $(V(G),\Omega_\alpha)$
constructed above.
Every $f:(G,\alpha)\to (G',\alpha')$ induces a $\PSTS$-morphism $F(f):F(G,\alpha)\to F(G',\alpha')$.
Indeed, for all $\{u,v,\alpha(\{u,v\})\}\in\Omega_\alpha$,
\begin{align*}
\{f(u),f(v),f(\alpha(\{u,v\}))\}&=\\
\{f(u),f(v),\alpha(S(f)(\{u,v\}))\}&=
\{f(u),f(v),\alpha(\{f(u),f(v)\})\},
\end{align*}
because $f$ is a morphism of algebras. $F$ is then a functor from $\Graph^T$ to $\PSTS$. 

On the other hand, let $(A,\Omega_A)$ be a partial Steiner triple system.
Let $G$ be a graph with $V(G)=A$ and $u\sim_G v$ if and only if $u\neq v$ and 
$u,v\in X$ for some $X\in\Omega_A$. Let us define a morphism of graphs 
$\alpha:S(G)\to G$ by the rules $\alpha(\{u\})=u$ and $\alpha(\{u,v\})=w$,
where $w$ is the unique element of $A$ such that $\{u,v,w\}\in\Omega_A$.

Clearly, $\alpha$ is a morphism of graphs. We need to prove that $(G,\alpha)$
is an algebra for the Steiner triple monad. The triangle axiom~\eqref{diag:algtriangle}
is clearly satisfied by $\alpha$, so
let us check the square axiom~\eqref{diag:algsquare}.
Let $K$ be a vertex of $S^2(G)$. We need to prove that $\alpha(\mu_G(K))=\alpha(S(\alpha)(K))$.
Let $u,v\in X$ for some $X\in\Omega_A$, so that $\{u,v\}$ is an edge of $G$. 

If $K=\{\{u,v\},\{u\}\}$, 
then
$\alpha(\mu_G(K))=\alpha(\{u,v\}\Delta\{v\})=\alpha(\{v\})=v$ and $\alpha(S(\alpha)(\{\{u,v\},\{u\}))=
\alpha(\{\alpha(\{u,v\}),u\})=v$, because $X=\{u,v,\alpha(\{u,v\})$.

The cases of $K=\{\{u,v\}\}$, $K=\{\{u\}\}$ and $K=\{\{u\},\{v\}\}$ are trivial are thus omitted.

To prove the functoriality of this construction,
let us write $U(A,\Omega_A)=(G,\alpha)$. We claim that every morphism
in $\PSTS$ $f:(A_1,\Omega_{A_1})\to(A_2,\Omega_{A_2})$ induces
a morphism $U(f):U(A_1,\Omega_{A_1})\to U(A_2,\Omega_{A_2})$. This amounts
to proving that~\eqref{diag:algmorphism} commutes.

If $\{u,v\}$ is an edge of $U(A_1,\Omega_{A_1})$, then
$\{u,v,\alpha_1(\{u,v\})\}\in \Omega_{A_1}$ is the unique triple
that contains $u$ and $v$. Since $f$ is a morphism of partial Steiner triple
systems, $\{f(u),f(v),f(\alpha_1(\{u,v\}))\}\in \Omega_{A_2}$. Further,
$\alpha_2(S(f)(\{u,v\}))=\alpha_2(\{f(u),f(v)\})$ and
$\{f(u),f(v),\alpha_2(\{f(u),f(v)\})\}$ is the unique triple that contains
$f(u)$ and $f(v)$. Therefore $f(\alpha_1(\{u,v\}))=\alpha_2(S(f)(\{u,v\}))$,
meaning that~\eqref{diag:algmorphism} commutes.

For a singleton vertex $\{u\}$ of $U(A_1,\Omega_{A_1})$, 
$$
f(\alpha_1(\{u\}))=f(u)=\alpha_2(\{f(u)\})=\alpha_2(S(f)(\{u\})).
$$

Therefore $U$ is a functor and it is easy to check that
$UF=\id_{\Graph^S}$ and $FU=\id_\PSTS$.
\end{proof}

\begin{corollary}
Let $(A,\Omega_A)$, $(B,\Omega_B)$ be partial Steiner triple systems. Then their product
in $\PSTS$ is $(A\times B,\Omega_{A\times B})$, where
$X\in\Omega_{A\times B}$ if and only if $X$ is a 3-element subset of
$A\times B$ such that $p_A(X)\in\Omega_A$ and $p_B(X)\in\Omega_B$.
\end{corollary}
\begin{proof}
We can describe the product in the
category $\Graph^S$ using the diagram~\eqref{diag:product}: the product
of two algebras $(A,\alpha)$ and $(B,\gamma)$ for the Steiner triple monad
is the algebra $(A\times B,\gamma)$, where $\gamma$ is given by the rules
\begin{align*}
\gamma(\{(a,b)\})&=(\alpha(\{a\}),\beta(\{b\}))=(a,b)\\
\gamma(\{(a_1,b_1),(a_2,b_2)\})&=(\alpha(\{a_1,a_2\}),\beta(\{b_1,b_2\})),
\end{align*}
for every vertex $(a,b)$ and edge $\{(a_1,b_1),(a_2,b_2)\}$ of the graph $A\times B$. 
That means
that the product of partial Steiner triple systems $(A,\Omega_A)$ and
$(B,\Omega_B)$ in $\PSTS$ is $(A\times B,\Omega_{A\times B})$ with
$$
\Omega_{A\times B}=\{\{(a_1,b_1),(a_2,b_2),\bigl(\alpha(\{a_1,a_2\}),\beta(\{b_1,b_2\})\bigr)\}:a_1\sim_A a_2\text{ and }b_1\sim_B b_2\},
$$
where $(A,\alpha)$ and $(B,\beta)$ are the algebras associated with $(A,\Omega_A)$ and $(B,\Omega_B)$, respectively.
The statement then easily follows.
\end{proof}

\section{Conclusion and further work}

We have shown that both the category of perfect matchings and the category of
partial Steiner triple systems can by represented as algebras for a monad.
There is a similarity between the monads: one can say that the monad $S$ does something
similar to the monad $T$, but one dimension higher. This suggests an obvious question,
whether these monads belong to some more general family of monads on $\Graph$. We plan to investigate
this question in a future paper.


\begin{thebibliography}{10}

\bibitem{awodey2006category}
\uppercase{Awodey, S.}
\newblock \textit{Category {T}heory.}
\newblock Oxford Logic Guides \textbf{49}, Oxford University Press, 2006.

\bibitem{brown2008graphs}
\uppercase{Brown, R.---Morris, I.---Shrimpton, J.,---Wensley, C.}
\newblock \textit{Graphs of morphisms of graphs.}
\newblock The Electronic Journal of Combinatorics \textbf{15} (2008)

\bibitem{Gra:UA}
\uppercase{Gr{\" a}tzer, G.}
\newblock \textit{Universal Algebra}, second~ed.
\newblock Springer-Verlag, 1979.

\bibitem{hahn1997graph}
\uppercase{Hahn, G.---Tardif, C.}
\newblock \textit{Graph homomorphisms: structure and symmetry.}
\newblock Graph symmetry, Springer, 1997, pp.~107--166.

\bibitem{hell2004graphs}
\uppercase{Hell, P.---Neset\v ril, J.}
\newblock \textit{Graphs and homomorphisms}.
\newblock Oxford University Press, 2004.

\bibitem{mac1998categories}
\uppercase{Lane, S.~M.}
\newblock \textit{Categories for the Working Mathematician}.
\newblock No.~5 in Graduate Texts in Mathematics. Springer-Verlag, 1971.

\bibitem{lovasz2009matching}
\uppercase{Lov{\'a}sz, L.---Plummer, M.~D.}
\newblock \textit{Matching theory}, vol.~367.
\newblock American Mathematical Soc., 2009.

\bibitem{maclane2012sheaves}
\uppercase{Mac~Lane, S.---Moerdijk, I.}
\newblock \textit{Sheaves in geometry and logic: A first introduction to topos
  theory}.
\newblock Springer Science \& Business Media, 2012.

\bibitem{manes1974compact}
\uppercase{Manes, E.}
\newblock \textit{Compact {H}ausdorff objects.}
\newblock General Topology and its Applications \textbf{4} (1974), 341--360.

\bibitem{moggi1991notions}
\uppercase{Moggi, E.}
\newblock \textit{Notions of computation and monads.}
\newblock Information and computation \textbf{93} (1991), 55--92.

\bibitem{pultr1980combinatorial}
\uppercase{Pultr. A---Trnkov\'a V}
\newblock \textit{Combinatorial, Algebraic and Topological Representations of Groups}
\newblock {North-Holland, Amsterdam, 1980.}

\bibitem{riehl2016category}
\uppercase{Riehl, E.}
\newblock \textit{Category theory in context}.
\newblock Courier Dover Publications, 2016.

\end{thebibliography}
\end{document}